\documentclass[10pt, one side,emlines]{amsart}
\usepackage{tikz}
\usepackage{amssymb,latexsym,xy,eucal,mathrsfs}
\usepackage{bm,amsfonts,amssymb,amsthm,newlfont,eucal, mathrsfs,latexsym, caption}
\textwidth=15.2cm \textheight=24cm \theoremstyle{plain}
\newtheorem{lem}{Lemma}[section]

\newtheorem{thm}[lem]{Theorem}
\newtheorem{prop}[lem]{Proposition}
\newtheorem{cor}[lem]{Corollary}
\theoremstyle{definition}

\newcommand{\Z}{{\mathbb{Z}}}

\begin{document}

	\baselineskip 14truept
	\title{Linear and group identifying codes in Hamming Graphs}
	
	\author{N. V. Shinde and S. A. Mane }
	
	\dedicatory{Department of Mathematics, COEP Technological University, Pune-411005, India.\\Center for Advanced Studies in Mathematics,
		Department of Mathematics,\\ Savitribai Phule Pune University, Pune-411007, India.\\
		 nvs.maths@coeptech.ac.in : manesmruti@yahoo.com  \\}
	
	\maketitle
	
	\begin{abstract}

Codes are crucial in many areas of applications. Different types of codes are designed to meet specific needs, which makes them more effective and useful.
Linear codes are extensively used in data storage systems. Identifying codes are essential for locating malfunctioning processors. To combine these benefits, researchers have looked into a type of code called linear identifying codes. These codes blend the error-correction abilities of linear codes with the fault-finding capabilities of identifying codes.
Group codes are also highly regarded for their strong properties and reliable decoding methods.
In our work, we introduce a new type of identifying code called group Identifying codes. These codes aim to bring together the best features of both Identifying codes and group codes, offering enhanced performance in fault detection and system reliability. In this paper, we establish limits on the smallest size of a group identifying code when \( G \) is an \( n \)-dimensional Hamming cube \( K_{m_1} \square K_{m_2} \square \dots \square K_{m_n} \). Additionally, we determine the smallest size of a linear identifying code in \( K_p^n \) for a prime \( p \) and \( n \geq 2 \). In [1], it was hypothesized that \( \gamma^{ID}(K_m^3) = m^2 \) for an integer \( m \geq 2 \). Although this conjecture was disproven in [2], we demonstrate that group identifying codes in \( K_m^3 \) for an integer \( m \geq 2 \) and linear identifying codes in \( K_p^3 \) for a prime \( p \) indeed fulfill this conjecture.

	\end{abstract}

	%\maketitle
	
	
	\section{Introduction} 
	
	In this paper, we focus on simple, connected, and undirected graphs. 
	Let $G$ be a graph and let $v \in V(G)$ be a vertex. For a subset $C \subseteq V(G)$, $
	 N_G[v] \cap C  = J_C(v)$ (say)
		where $N_G[v]$ denotes the closed neighborhood of $v$ in $G$.

	 If the set \( J_C(v) \) is non-empty and distinct for all vertices in \( G \), then \( C \) is called an identifying code. The minimum cardinality of an identifying code in $G$ is denoted by $\gamma^{ID}(G)$. Identifying codes can only exist in a graph if the graph is twin-free, meaning no two distinct vertices have exactly the same set of neighbors.
	
	 Codes play a pivotal role across a wide array of applications, with various types tailored to address specific requirements, thereby increasing their effectiveness and utility. Linear codes \cite{ha, luo}, for example, are widely utilized in data storage systems due to their robustness in error correction. Identifying codes, on the other hand, are indispensable for pinpointing malfunctioning processors. In an effort to harness the strengths of both, researchers have explored a novel category of code known as linear identifying codes\cite{ra, ran}, which seamlessly integrate the error-correction prowess of linear codes with the fault-detection capabilities of identifying codes.
	
	Group codes\cite{be, ga, pi, pil, pill} are similarly esteemed for their strong structural properties and reliable decoding mechanisms. Building on this foundation, our work introduces a new class of identifying codes—group identifying codes. These codes are designed to merge the key advantages of identifying codes and group codes, providing an elevated level of fault detection and enhancing system reliability. 
	
Thus, an identifying code of a graph \( G \) is a code that is both dominating and separating. If \( V(G) \) is a group, a group identifying code is an identifying code that forms a subgroup. If \( V(G) \) is a vector space, a linear identifying code is an identifying code that forms a subspace. The minimum size of an identifying code (group identifying code, linear identifying code) in \( G \) is denoted as \( \gamma^{ID}(G) \) (\( \gamma^{GID}(G) \), \( \gamma^{LID}(G) \), respectively). Since a group or linear identifying code is also an identifying code, \( \gamma^{GID}(G) \geq \gamma^{ID}(G) \) (\( \gamma^{LID}(G) \geq \gamma^{ID}(G) \)).
	
	An \( n \)-dimensional Hamming cube is the product of \( n \) complete graphs \( K_{m_1} \square K_{m_2} \square \cdots \square K_{m_n} \), where each \( K_m \) is a complete graph on \( m \) vertices. The Hamming distance \( d_H(u, v) \) between two vertices \( u \) and \( v \) is the number of edges between them in the cube. The Hamming weight \( wt_H(u) \) of a vertex \( u \) is the number of non-zero coordinates in \( u \). Two vertices are adjacent if they differ in exactly one coordinate. Thus, we consider the $n$-dimensional Hamming cube $K_{m_1} \square \dots \square K_{m_n}$, which can be viewed as a graph on the group $\mathbb{Z}_{m_1} \times \dots \times \mathbb{Z}_{m_n}$, where two vertices are adjacent if they differ in exactly one coordinate. If all \( m_i = m \), the graph is called a \( K_m^n \), or an \( n \)-dimensional \( m \)-ary Hamming cube. The vertices of the graph form an abelian group under componentwise addition, where each component uses modular addition. The Hamming distance between two vertices \( u \) and \( x \) is the Hamming weight of \( u - x \).
	
		In 1998, Karpovski et al.\cite{k}. introduced identifying codes in graphs. They studied these codes in structures like the Binary Hypercube, Non-binary Lee cubes, and other topologies such as trees, hexagonal mesh, and triangular mesh, Many researchers\cite{bla, blas, cha, ex, exo, k,ra, ran} have since explored identifying codes in binary hypercubes.
	
	In 2008, Gravier et al.\cite{gr} found the exact value of the identifying code \( \gamma^{ID}(K_m \square K_m) = \lceil \frac{3m}{2} \rceil \) for \( m \geq 2 \). In 2013, Goddard and Wash\cite{go} proved that for \( n \leq m \), \( \gamma^{ID}(K_n \square K_m) = \max \{ 2m - n, m + \lfloor \frac{n}{2} \rfloor \} \). They also showed that \( \gamma^{ID}(K_m^d) \leq m^{d-1} \) for \( d \geq 3 \), and \( \gamma^{ID}(K_m^3) \geq m^2 - m\sqrt{m} \). They conjectured that \( \gamma^{ID}(K_m^3) \geq m^2 \) for all \( m \geq 1 \). However, in 2019, Junnila et al.\cite{jun} disproved this conjecture by showing that \( \gamma^{ID}(K_m^3) \leq m^2 - \frac{m}{4} \) when \( m \) is a power of four, and improved the lower bound to \( \gamma^{ID}(K_m^3) \geq m^2 - \frac{3}{2}m \). They found that the conjecture holds for a class of codes called self-locating-dominating codes.
	
	In 2018\cite{junn}, self-identifying codes were introduced, where a code \( C \) in a graph \( G \) is self-identifying if it's identifying and for all \( u \in V \) and \( U \subseteq V(G) \), \( |U| \geq 2 \) and \( J_C(u) \neq J_C(U) \). In 2019\cite{jun}, self-locating-dominating codes were characterized, where a code \( C \) is self-locating-dominating in \( G \) if for every \( u \in V(G) - C \), \( J_C(u) \neq \emptyset \) and \( \bigcap_{c \in J_C(u)} N[c] = \{u\} \).
	
	Theorems on self-locating-dominating codes and self-identifying codes are provided in the text.
	\begin{thm}[\cite{junn}]\label{2j1} A code $D$ is a self-locating-dominating code in a graph $G$ if and only if for each non-codeword $x$ and $y\in V(G)-\{x\}$, $J_D(x)-J_D(y)\neq \emptyset$. 	\end{thm}
	In \cite{jun}, self-identifying codes are charaterised as follows.
	\begin{thm}[\cite{jun}]\label{2j2} Let $q$ be a prime power. A code $D$ is a self-identifying code in a graph $\Z_q^n$ if and only if for each word $x$, $|J_D(x)|\geq 3$ and there exists $d,~d'\in J_D(x)$ such that $d(d,d')=2$. \end{thm}
	 In this paper we discuss two classes of codes—group identifying codes and linear identifying codes—which are different from self-identifying and self-locating-dominating codes, and for which the conjecture holds.
	
Here’s an example where a code is both self-identifying and self-locating-dominating, but not a group or linear identifying code, showing the difference between these types of codes.

Consider \( G = \mathbb{Z}_2^3 \). 

- The code \( D \) in Figure 1(a) is a group (linear) identifying code. However, since \( J_D((1,1,1)) - J_D((0,1,1)) = \emptyset \), this code is not self-locating-dominating (based on Theorem 1.1). Also, since \( |J_D((1,1,1))| = 1 \), this code is not self-identifying (based on Theorem 1.2).

- The code \( D \) in Figure 1(b) is both self-locating-dominating and self-identifying (based on Theorems 1.1 and 1.2). However, since \( (1,0,0) \) and \( (0,1,0) \) are in \( D \), but \( (1,0,0) + (0,1,0) = (1,1,0) \) is not in \( D \), this code is neither a group nor a linear identifying code.\\\\

	\begin{center}
		$\begin{tikzpicture} [scale=0.85]
			\draw (8,-5)--(6,-5)--(6,-3)--(8,-3)--(8,-5);
			\draw (6,-5)--(7,-4)--(7,-2)--(6,-3);
			\draw  (7,-4)--(9,-4)--(8,-5);
			\draw  (9,-4)--(9,-2)--(7,-2);
			\draw  (9,-2)--(8,-3);
			\draw[fill](6,-5) circle(.1);
			\draw[fill](6,-3) circle(.1);
			\draw[fill](7,-4) circle(.1);
			\draw[fill](8,-5) circle(.1);
			\draw[fill](8,-3) circle(.1);
			\draw[fill](7,-2) circle(.1);
			\draw[fill=white](9,-4) circle(.1);
			\draw[fill](9,-2) circle(.1);
			\node [below] at (6,-5){\tiny{$(0,0,0)$}};
			\node [above left] at (6,-3.3){\tiny{$(0,0,1)$}};
			\node [below] at (8,-5){\tiny{$(0,1,0)$}};
			\node [above left] at (8.3,-3.1){\tiny{$(0,1,1)$}};
			\node [below right] at (6.8,-4){\tiny{$(1,0,0)$}};
			\node [above right] at (6.8,-2){\tiny{$(1,0,1)$}};
			\node [below right] at (8.8,-4){\tiny{$(1,1,0)$}};
			\node [above right] at (8.8,-2){\tiny{$(1,1,1)$}};
			\node [below] at (7,-5.5){Figure 1(b)};

			\draw (2,-5)--(0,-5)--(0,-3)--(2,-3)--(2,-5);
			\draw (0,-5)--(1,-4)--(1,-2)--(0,-3);
			\draw  (1,-4)--(3,-4)--(2,-5);
			\draw  (3,-4)--(3,-2)--(1,-2);
			\draw  (3,-2)--(2,-3);
			\draw[fill](0,-5) circle(.1);
			\draw[fill](0,-3) circle(.1);
			\draw[fill=white](1,-4) circle(.1);
			\draw[fill](2,-5) circle(.1);
			\draw[fill](2,-3) circle(.1);
			\draw[fill=white](1,-2) circle(.1);
			\draw[fill=white](3,-4) circle(.1);
			\draw[fill=white](3,-2) circle(.1);
			
			\node [below] at (0,-5){\tiny{$(0,0,0)$}};
			\node [above left] at (0,-3.3){\tiny{$(0,0,1)$}};
			\node [below] at (2,-5){\tiny{$(0,1,0)$}};
			\node [above left] at (2.3,-3.1){\tiny{$(0,1,1)$}};
			\node [below right] at (.8,-4){\tiny{$(1,0,0)$}};
			\node [above right] at (.8,-2){\tiny{$(1,0,1)$}};
			\node [below right] at (2.8,-4){\tiny{$(1,1,0)$}};
			\node [above right] at (2.8,-2){\tiny{$(1,1,1)$}};
			\node [below] at (1,-5.5){Figure 1(a)};
			\node [below] at (4.5,-6.2) {Examples showing that self-locating-dominating codes and  };
			\node [below] at (4.5,-6.9) {self-identifying codes are different from group (linear) identifying codes};
			\node [below] at (4.5,-7.6) {(filled circles are codewords) };
		\end{tikzpicture}$
	\end{center}

In 2003, Ranto \cite{ra} studied binary linear identifying codes and proved that for any integer $r \geq 1$, if $3(2^r - 1) \leq n \leq 3(2^{r+1} - 1) - 1$, then the identifying code number of the Hamming graph $K_2^n$ is given by $\gamma^{\text{LID}}(K_2^n) = 2^{n - r}$. Further investigations into binary $r$-identifying codes were presented in \cite{ran}.

This work extends the study of identifying codes to the non-binary setting by considering group (or linear) codes over $\mathbb{Z}_p^n$, where $p > 2$ is a prime. In particular, we construct optimal linear identifying codes in $\mathbb{Z}_p^n$, and show that for every integer $q \geq 2$, the optimal group identifying codes in $\mathbb{Z}_q^3$ have cardinality $q^2$. These constructions generalize the linear identifying codes in $\mathbb{Z}_2^n$ introduced in \cite{ge} and form part of the broader study of identification problems in Hamming graphs.

	\section{\textbf{Preliminaries}}

	Throughout this paper, we assume that \( K_m \) is a complete graph with \( m \geq 2 \) vertices. We use \( d(u,v) \) to denote the\textit{ Hamming distance} between two vertices \( u \) and \( v \), and \( wt(u) \) for the \textit{Hamming weight} of a vertex \( u \). The \textit{zero vector} in the graph \( G \) is denoted by \( \mathbf{0} \).
	
	For a vertex \( u = (u_1, u_2, \ldots, u_n) \in V(G) \), its \textit{additive inverse} is defined as  
	\[ -u = (-u_1, -u_2, \ldots, -u_n) = (m_1 - u_1, m_2 - u_2, \ldots, m_n - u_n). \]
	
	We define the vector \( e_i^j \) as the vector with 0 in all positions except the \( i^{th} \) position, which is \( j \):  
	\[ e_i^j = (0, \ldots, 0, \underbrace{j}_{i^{th} \text{ place}}, 0, \ldots, 0), \]
	for \( 1 \leq i \leq n \) and \( 0 \leq j \leq m_i - 1 \). Then, its inverse is \( -e_i^j = e_i^{m_i - j} \). In particular, \( e_i^0 = \mathbf{0} \) for all \( i \). See Figure $2$\\\\
	
	\begin{center}
	\begin{tikzpicture}[scale=1.1]
		\draw (0,0) -- (0,3) --  (0,6) ;\draw (0,0) -- (3,0) --  (6,0) ;
		\draw (1,-1)--  (1,2)  -- (1,5) ;	\draw (2,-2) --  (2,1) -- (2,4) ;	\draw  (3,0)--   (3,3)--  (3,6);
		\draw  (0,3)--   (3,3)--  (6,3);	\draw (4,-1) -- (4,2) -- (4,5);	\draw (5,-2)--  (5,1) --  (5,4) ;	\draw   (6,0)-- (6,3)-- (6,6);	\draw   (0,6)-- (3,6)-- (6,6);	\draw (7,-1) --  (7,2) --  (7,5) ;	\draw (8,-2) -- (8,1) --  (8,4) ;
		
		\draw (0,0)--(1,-1)--  (2,-2)  ;\draw (0,3) --  (1,2)--(2,1) ;\draw (0,6) -- (1,5) -- (2,4);
		\draw (3,0)--(4,-1)--  (5,-2)  ;\draw (3,3) --  (4,2)--(5,1) ;\draw (3,6) -- (4,5) -- (5,4);
		\draw (6,0)--(7,-1)--  (8,-2)  ;\draw (6,3) --  (7,2)--(8,1) ;\draw (6,6) -- (7,5) -- (8,4);
		\draw (1,-1)--(4,-1)--(7,-1);\draw (1,2)--(4,2)--(7,2);\draw (1,5)--(4,5)--(7,5);
		\draw (2,-2)--(5,-2)--(8,-2);\draw(2,1)--(5,1)--(8,1) ;\draw (2,4)--(5,4)--(8,4);
		
		\draw (0,0).. controls (0.3,3)..(0,6);\draw (1,-1).. controls (1.3,3)..(1,5);\draw (2,-2).. controls (2.3,2)..(2,4);
		\draw (3,0).. controls (3.3,3)..(3,6);\draw (4,-1).. controls (4.3,2)..(4,5);\draw (5,-2).. controls (5.3,2)..(5,4);
		\draw (6,0).. controls (6.3,3)..(6,6);\draw (7,-1).. controls (7.3,2)..(7,5);\draw (8,-2).. controls (8.3,2)..(8,4);
		\draw (0,0)..controls(1.5,-1)..  (2,-2)  ;\draw (3,0)..controls(4.5,-1)..  (5,-2);\draw (6,0)..controls(7.5,-1)..  (8,-2);
		\draw (0,3) ..controls (1.5,2)..(2,1) ;\draw (3,3) ..controls  (4.5,2)..(5,1) ;\draw (6,3) ..controls  (7.5,2)..(8,1) ;
		\draw (0,6) ..controls (1.5,5) .. (2,4);\draw (3,6) ..controls(4.5,5) .. (5,4);\draw (6,6) ..controls (7.5,5) .. (8,4);
		
		\draw (0,0)..controls(5.5,0.3)..(6,0);\draw (0,3)..controls(5.5,3.4)..(6,3);\draw (0,6)..controls(5.5,6.3)..(6,6);
		\draw (1,-1)..controls(5.5,-0.7)..(7,-1);\draw (1,2)..controls(5.5,2.5)..(7,2);\draw (1,5)..controls(5.5,5.5)..(7,5);
		\draw (2,-2)..controls(5.5,-1.6)..(8,-2);\draw (2,1)..controls(5.5,1.6)..(8,1);\draw (2,4)..controls(5.5,4.6)..(8,4);
		
		\draw [fill](0,0) circle [radius=0.1];\draw[fill=white]  (0,3) circle [radius=0.1];\draw [fill=white](0,6) circle [radius=0.1];
		
		\draw (0,-0.2)node[left]{\tiny$(0,0,0)$};\draw (0,2.8)node[left]{\tiny$e_3^1=(0,0,1)$};\draw (0,5.8)node[left]{\tiny$e_3^2=(0,0,2)$};
		\draw[fill=white] (1,-1) circle [radius=0.1];\draw [fill] (1,2) circle [radius=0.1];\draw [fill=white] (1,5) circle [radius=0.1];
		
		\draw (1,-1.2)node[left]{\tiny$e_1^1=(1,0,0)$};\draw (1,1.8)node[left]{\tiny$(1,0,1)$};\draw (1,4.8)node[left]{\tiny$(1,0,2)$};
		\draw [fill=white](2,-2) circle [radius=0.1];\draw [fill=white] (2,1) circle [radius=0.1];\draw [fill](2,4) circle [radius=0.1];
		
		\draw (2,-2.2)node[left]{\tiny$e_1^2=(2,0,0)$};\draw (2,0.8)node[left]{\tiny$(2,0,1)$};\draw (2,3.8)node[left]{\tiny$(2,0,2)$};
		\draw [fill=white] (3,0) circle [radius=0.1];\draw  [fill] (3,3)circle [radius=0.1];\draw [fill=white] (3,6)circle [radius=0.1];
		
		\draw (3,-0.2)node[left]{\tiny$(0,1,0)$};\draw (3,2.8)node[left]{\tiny$(0,1,1)$};\draw (3,5.8)node[left]{\tiny$(0,1,2)$};
		\draw (3,.3)node[left]{\tiny$e_2^1$};\draw (3,6.4)node[left]{\tiny$e_2^1+e_3^2$};
		\draw [fill=white](4,-1) circle [radius=0.1];\draw[fill=white] (4,2) circle [radius=0.1];\draw  [fill](4,5) circle [radius=0.1];
		
		\draw (4,-1.2)node[left]{\tiny$(1,1,0)$};\draw (4,1.8)node[left]{\tiny$(1,1,1)$};\draw (4,4.8)node[left]{\tiny$(1,1,2)$};
		\draw [fill](5,-2) circle [radius=0.1];\draw [fill=white] (5,1) circle [radius=0.1];\draw [fill=white] (5,4) circle [radius=0.1];
		
		\draw (5,-2.2)node[left]{\tiny$(2,1,0)$};\draw (5,0.8)node[left]{\tiny$(2,1,1)$};\draw (5,3.8)node[left]{\tiny$(2,1,2)$};
		\draw  [fill=white] (6,0) circle [radius=0.1];\draw[fill=white]  (6,3) circle [radius=0.1];\draw  [fill](6,6) circle [radius=0.1];
		
		\draw (6,-0.2)node[left]{\tiny$(0,2,0)$};\draw (6,2.8)node[left]{\tiny$(0,2,1)$};\draw (6,5.8)node[left]{\tiny$(0,2,2)$};
		\draw (6,.3)node[left]{\small$e_2^2$};
		\draw [fill](7,-1) circle [radius=0.1];\draw [fill=white] (7,2) circle [radius=0.1];\draw [fill=white] (7,5) circle [radius=0.1];
		
		\draw (7,-1.2)node[left]{\tiny$(1,2,0)$};\draw (7,1.8)node[left]{\tiny$(1,2,1)$};\draw (7,4.8)node[left]{\tiny$(1,2,2)$};
		\draw [fill=white](8,-2) circle [radius=0.1];\draw  [fill](8,1) circle [radius=0.1];\draw [fill=white] (8,4) circle [radius=0.1];
		
		\draw (8,-2.2)node[left]{\tiny$(2,2,0)$};\draw (8,0.8)node[left]{\tiny$(2,2,1)$};\draw (8,3.8)node[left]{\tiny$(2,2,2)$};

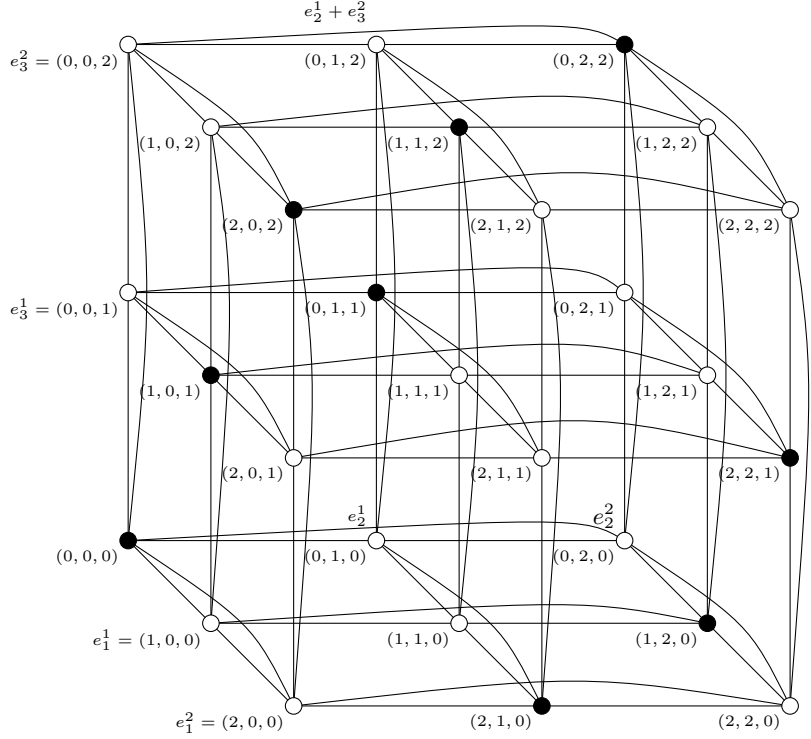
		\draw [thick] (4,-3) node[below]{Figure 2: A group (linear) identifying code in the Hamming graph $ K_3^3$ } ;\draw [thick] (4,-3.5) node[below]{(filled circles are codewords)} ;
		
	\end{tikzpicture}
\end{center}
	
	The \textit{open neighborhood} of a vertex \( u \in V(G) \) is given by:  
	\[ N(u) = \{ u + e_i^j : 1 \leq i \leq n, 1 \leq j \leq m_i - 1 \}. \]
	
	For a set \( C \subseteq V(G) \) and a vertex \( u \notin C \), the \textit{coset} of \( C \) with respect to \( u \) is  
	\[ u + C = \{ u + c : c \in C \}, \]
	and the \textit{distance from \( u \) to \( C \)} is  
	\[ d(u, C) = \min \{ d(u, c) : c \in C \}. \]
	
	If \( C_1 \) and \( C_2 \) are two sets, their \textit{direct sum} is defined as:  
	\[ C_1 \oplus C_2 = \{ (c_1, c_2) : c_1 \in C_1, c_2 \in C_2 \}. \]

	The following lemma was first proven by Junnila et al. \cite{ju} for the Hamming graph \( G = K_q^n \), where \( q \) is a prime power greater than 2 and \( n > 3 \). We have found that the result also holds for the more general graph \( G = K_{m_1} \square K_{m_2} \square \cdots \square K_{m_n} \), where each \( m_i \geq 2 \). However, since the proof is different in this case, we provide it here.

	\begin{lem}\label{2tm} 
	Let \( G = K_{m_1} \square K_{m_2} \square \cdots \square K_{m_n} \) for \( n \geq 2 \), and let \( u, v \in G \). Then:
	
	\[
	|N[u] \cap N[v]| = 
	\begin{cases}
		m_i & \text{if } d(u, v) = 1 \text{ and } u, v \text{ differ at the } i^{th} \text{ coordinate}, \\
		2 & \text{if } d(u, v) = 2, \\
		0 & \text{if } d(u, v) > 2.
	\end{cases}
	\]\end{lem}
	
	\begin{proof}  
	- If \( d(u, v) = 1 \), and they differ only at the \( i^{th} \) coordinate, then \( v = u + e_i^j \) for some \( 1 \leq j \leq m_i - 1 \). For any \( s \neq i \), \( u + e_s^t \) (with \( t \neq 0 \)) is at distance 2 from \( v \), so it cannot be in both neighborhoods. Hence,  
	\[ N[u] \cap N[v] = \{ u + e_i^j : 0 \leq j \leq m_i - 1 \}, \]
	and the size of the intersection is \( m_i \).
	
	- If \( d(u, v) = 2 \), then \( v = u + e_{i_1}^{j_1} + e_{i_2}^{j_2} \) for distinct coordinates \( i_1 \neq i_2 \). Then, the only common neighbors of \( u \) and \( v \) are \( u + e_{i_1}^{j_1} \) and \( u + e_{i_2}^{j_2} \). So the intersection has size 2.
	
	- If \( d(u, v) > 2 \), assume for contradiction that there is a vertex \( w \in N[u] \cap N[v] \). Then,  
	\[ u = w + e_{i_1}^{j_1}, \quad v = w + e_{i_2}^{j_2}, \]
	with \( i_1 \neq i_2 \) and \( j_1, j_2 \neq 0 \). This would imply \( d(u, v) = d(e_{i_1}^{j_1}, e_{i_2}^{j_2}) = 2 \), which contradicts \( d(u, v) > 2 \). So,  
	\[ N[u] \cap N[v] = \emptyset. \]
	
\end{proof}

\section{Group identifying codes in $K_{m_1}\square K_{m_2}\square \cdots \square K_{m_n}$} 
	
	In this section, we study \textit{group identifying codes} in the graph \( G = K_{m_1} \square K_{m_2} \square \cdots \square K_{m_n} \) for \( n \geq 2 \). In \cite{go} and \cite{gr}, the value of \( \gamma^{ID}(G) \) was determined for \( n = 2 \), which serves as a lower bound for \( \gamma^{GID}(G) \) in that case. For \( n \geq 3 \), we establish a new lower bound for \( \gamma^{GID}(G) \).
	
	It is also worth noting that the full vertex set \( V(G) \) forms a group identifying code in \( G \) for any \( n \geq 2 \). This provides a natural upper bound for \( \gamma^{GID}(G) \), which may be useful in further analysis.

	Before proving the main theorem of this section, we first establish some necessary and sufficient conditions that will be used in its proof.

	A lemma similar to the one below was proven by Ranto \cite{ra} for the graph \( G = K_2^n \), where \( n \geq 2 \). We have found that the result also holds for the  graph \( G = K_{m_1} \square K_{m_2} \square \cdots \square K_{m_n} \), with \( n \geq 2 \). Since the reasoning is very similar, we omit the proof here.

	\begin{lem}\label{2lm1}  
	Let \( n \geq 2 \). If \( C \) is a group code in the graph  
	\[ G = K_{m_1} \square K_{m_2} \square \cdots \square K_{m_n}, \]  
	and \( c \in C \), then for any \( u \in V(G) \), we have  
	\[ J_C(u + c) = J_C(u) + c. \]
\end{lem}

	Before proving the sufficient condition, we first need to prove the following lemma.
	
	\begin{lem}\label{2lm2}  
	Let \( n \geq 2 \). Suppose \( C \) is a group code in the graph  
	\[ G = K_{m_1} \square K_{m_2} \square \cdots \square K_{m_n}, \]  
	and let \( u, v \in V(G) \). If either \( u \) and \( v \) are in the same coset of \( C \), or \( u + v \in C \), then  
	\[ |J_C(u)| = |J_C(v)|. \]
	\end{lem}
\begin{proof}  
	It is a known fact that two vertices \( u \) and \( v \) are in the same coset of \( C \), say \( C + x \), if and only if we can write \( u = c_1 + x \) and \( v = c_2 + x \), for some \( c_1, c_2 \in C \). This is equivalent to saying that \( u - v = c_1 - c_2 \in C \).
	
\textbf{	Case 1:} Suppose \( u \) and \( v \) are in the same coset of \( C \).  
	Then \( u - v \in C \).  
	By Lemma \ref{2lm1}, we have:  
	\[
	J_C(u) = J_C(v + (u - v)) = J_C(v) + (u - v).
	\]  
	So, \( |J_C(u)| = |J_C(v)| \).
	
	\textbf{Case 2:} Suppose \( u + v \in C \).  
	Then \( u \) and \( -v \) are in the same coset of \( C \), since \( u - (-v) = u + v \in C \).  
	By Case 1, it follows that \( |J_C(u)| = |J_C(-v)| \).
	
	Since \( C \) is a group code, for every \( c \in C \), the element \( -c \) is also in \( C \).  
	Now, take any \( w \in J_C(v) \). Then \( w = v + e_i^j \) for some \( i, j \).  
	So,  
	\[
	-w = -v - e_i^j = -v + e_i^{(m_i - j)}.
	\]  
	This means \( w \in J_C(v) \) if and only if \( -w \in J_C(-v) \).  
	Thus, \( |J_C(v)| = |J_C(-v)| \), and hence  
	\[
	|J_C(u)| = |J_C(v)|.
	\]
\end{proof}

	Now, we give a sufficient condition for a subset of the vertex set \( V(G) \) to be an identifying code in \( G \). This condition is helpful in proving the existence of a group identifying code in \( G \), although it is not a necessary condition (see Figure 1(a) for a counterexample).
	
\begin{prop}\label{2lm3}  
	Let \( C \subseteq V(G) \), where  
	\[
	G = K_{m_1} \square K_{m_2} \square \cdots \square K_{m_n}, \quad n \geq 2, \quad m_i \geq 2 \text{ for all } 1 \leq i \leq n.
	\]  
	If for every vertex \( u \in V(G) \), the set \( J_C(u) \) satisfies \( |J_C(u)| \geq 3 \), and there exist two distinct indices \( i_1, i_2 \in \{1, \dots, n\} \) such that  
	\[
	\{u + e_{i_1}^{j_1}, u + e_{i_2}^{j_2}\} \subseteq C
	\]  
	for some \( 1 \leq j_1 \leq m_{i_1} - 1 \) and \( 1 \leq j_2 \leq m_{i_2} - 1 \), then \( C \) is an identifying code in \( G \).
	
\end{prop}
\begin{proof}  
	Since each vertex is adjacent to at least three codewords, \( C \) is a dominating set.
	
	To prove that \( C \) is an identifying code, we need to show that for any two distinct vertices \( u, v \in V(G) \), the sets \( J_C(u) \) and \( J_C(v) \) are different.
	
	We consider three cases based on the distance between \( u \) and \( v \).
	
	---
	
\textbf{	Case 1: \( d(u, v) = 1 \)}  
	Then \( v = u + e_{i_1}^{j_1} \) for some \( i_1 \) and \( j_1 \ne 0 \).  
	By assumption, there is another index \( i_2 \ne i_1 \) and \( j_2 \ne 0 \) such that \( u + e_{i_2}^{j_2} \in C \).  
	
	Now, the distance between \( v \) and this codeword is:
	\[
	d(v, u + e_{i_2}^{j_2}) = 2,
	\]  
	so this codeword is in \( J_C(u) \), but not in \( J_C(v) \).  
	Hence, \( J_C(u) \ne J_C(v) \).
	
	---
	
	\textbf{Case 2: \( d(u, v) = 2 \)}  
	If either \( u \in C \) or \( v \in C \), then one of them appears in its own closed neighborhood but not in the other’s, so \( J_C(u) \ne J_C(v) \).
	
	Now suppose both \( u, v \notin C \). Then:
	\[
	v = u + e_{i_1}^{j_1} + e_{i_2}^{j_2}, \quad \text{with } i_1 \ne i_2, \, j_1, j_2 \ne 0.
	\]  
	From Lemma \ref{2tm}, we know:
	\[
	N[u] \cap N[v] = \{u + e_{i_1}^{j_1}, u + e_{i_2}^{j_2}\}.
	\]  
	By assumption, there exists a third index \( i_3 \notin \{i_1, i_2\} \) and \( j_3 \ne 0 \) such that \( u + e_{i_3}^{j_3} \in C \).
	
	Then, the distance between this codeword and \( v \) is 3, so it lies in \( J_C(u) \) but not in \( J_C(v) \). Therefore, \( J_C(u) \ne J_C(v) \).
	
	If instead \( i_3 = i_1 \) or \( i_3 = i_2 \), but \( j_3 \ne j_1 \) or \( j_3 \ne j_2 \), then again the distance to \( v \) will be 2, so this codeword lies in \( J_C(u) \) but not in \( J_C(v) \). So again, \( J_C(u) \ne J_C(v) \).
	
	---
	
	\textbf{Case 3: \( d(u, v) > 2 \)}  
	By Lemma \ref{2tm}, \( N[u] \cap N[v] = \emptyset \), so their identifying sets are completely disjoint.  
	Clearly, \( J_C(u) \ne J_C(v) \).

	This completes the proof that \( C \) is an identifying code in \( G \).
\end{proof}

	 Using the above proposition, it is now straightforward to show that there exists a group identifying code in the graph  
	\[
	G = K_{m_1} \square K_{m_2} \square \cdots \square K_{m_n}.
	\]

\begin{cor}\label{2lm4} 
If \( n \geq 2 \) and \( m_i \geq 2 \) for all \( 1 \leq i \leq n \), then there exists a group identifying code in the graph  
\[
G = K_{m_1} \square K_{m_2} \square \cdots \square K_{m_n}.
\].\end{cor}

\begin{proof}  
The vertex set \( V(G) \) forms a group under addition.  
By Proposition \(\ref{2lm3}\), if \( n \geq 2 \), then \( V(G) \) is an identifying code in \( G \).  
Since \( V(G) \) is both a group and an identifying code, it is a group identifying code in \( G \) for \( n \geq 2 \).  
\end{proof}

Recall that a subgroup \( H \) of a group \( L \) is said to be proper if \( H \neq \{ \mathbf{0} \} \) and \( H \neq L \).

We now show that the graph  
\[
G = K_{m_1} \square K_{m_2}
\]  
admits no identifying code that is a proper subgroup of the vertex set, when \( m_1, m_2 \geq 2 \) and \( \gcd(m_1, m_2) = 1 \); that is, when \( m_1 \) and \( m_2 \) are relatively prime.

\begin{prop}
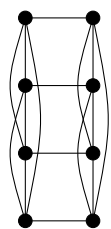
	There is no proper group identifying code in the graph \( G = K_{m_1} \square K_{m_2} \) when \( m_1 \geq 2 \) and \( m_2 \geq 2 \) are relatively prime integers.
\end{prop}

\begin{proof}
	The vertex set of \( G \) is given by \( V(G) = \mathbb{Z}_{m_1} \times \mathbb{Z}_{m_2} \), which is isomorphic to the cyclic group \( \mathbb{Z}_{m_1 m_2} \) when \( \gcd(m_1, m_2) = 1 \). Since \( \mathbb{Z}_{m_1 m_2} \) is a finite cyclic group of order \( m_1 m_2 \), it has a unique subgroup of order \( s \) for each divisor \( s \) of \( m_1 m_2 \), and every subgroup of a cyclic group is itself cyclic.
	
	Let \( H \leq V(G) \) be a subgroup generated by an element \( (s_1, s_2) \in \mathbb{Z}_{m_1} \times \mathbb{Z}_{m_2} \), where each \( s_i \) divides \( m_i \) for \( i = 1, 2 \). Since \( m_1 \) and \( m_2 \) are relatively prime, any nonzero divisors \( s_1, s_2 \) must also be relatively prime.
	
	We now analyze the structure of \( H \) based on the generator \( (s_1, s_2) \), and determine whether \( H \) forms an identifying code.
	
	\medskip
	\noindent
	\textbf{Case 1:} \( (s_1, s_2) = (0, 0) \).  
	Then \( H = \{ \mathbf{0} \} \), which is not a dominating set and hence not an identifying code.
	
	\medskip
	\noindent
	\textbf{Case 2:} \( (s_1, s_2) = (1, 0) \).  
	Then \( H = \mathbb{Z}_{m_1} \times \{0\} \). In this case, for all \( i \in \mathbb{Z}_{m_1} \),
	\[
	J_H((i, 0)) = \{(0,0), (1,0), \ldots, (m_1 - 1, 0)\},
	\]
	so all vertices in the row \( \mathbb{Z}_{m_1} \times \{0\} \) share the same identifying set. Thus, \( H \) is not separating and hence not an identifying code (see Figure~3(a)).
	
	\medskip
	\noindent
	\textbf{Case 3:} \( (s_1, s_2) = (1,1) \).  
	Then \( H = \mathbb{Z}_{m_1} \times \mathbb{Z}_{m_2} = V(G) \). This is the trivial case where the entire vertex set is the code, which is identifying by Corollary~\ref{2lm4}, but it is not a proper subgroup.
	
	\medskip
	\noindent
	\textbf{Case 4:} \( (s_1, s_2) = (1, s_2) \), where \( s_2 \in \mathbb{Z}_{m_2} \setminus \{0,1\} \) is a proper divisor of \( m_2 \).  
	Then \( H = \mathbb{Z}_{m_1} \times \mathbb{Z}_{s_2} \). In this case, the layers \( \mathbb{Z}_{m_1} \times \{1\} \) and \( \mathbb{Z}_{m_1} \times \{m_2 - 1\} \) are disjoint from \( H \), so vertices such as \( (0,1) \) and \( (0, m_2 - 1) \) satisfy
	\[
	J_H((0,1)) = J_H((0, m_2 - 1)),
	\]
	which implies \( H \) is not separating and hence not an identifying code (see Figure~3(b)).
		\begin{center}
		\begin{tikzpicture}
			\draw(0,0)--(1,0);\draw(0,1)--(1,1);\draw(0,2)--(1,2);
			\draw(0,0)--(0,2);\draw(1,0)--(1,2);
			\draw (0,0).. controls (-0.3,1)..(0,2);\draw (1,0).. controls (1.3,1)..(1,2);
			\draw[fill](0,0)circle[radius=0.1];\draw[fill](0,1)circle[radius=0.1];\draw[fill](0,2)circle[radius=0.1];
			\draw[fill=white](1,0)circle[radius=0.1];\draw[fill=white](1,1)circle[radius=0.1];\draw[fill=white](1,2)circle[radius=0.1];
			\draw [thick] (0,-1) node[below]{Figure 3(a): A subgroup } ;
			\draw [thick] (0,-1.7) node[below]{$H=\Z_3\times \{\bm 0\}$} ;

			\draw(5,0)--(5,2);\draw(6,0)--(6,2);\draw(7,0)--(7,2);\draw(8,0)--(8,2);
			\draw(5,0)--(8,0);\draw(5,1)--(8,1);\draw(5,2)--(8,2);
			\draw(5,0).. controls (6.5,0.3)..(8,0);\draw(5,1).. controls (6.5,1.3)..(8,1);\draw(5,2).. controls (6.5,2.3)..(8,2);
			\draw(5,0).. controls (5.5,-0.3)..(7,0);\draw(5,1).. controls (5.5,0.7)..(7,1);\draw(5,2).. controls (5.5,1.7)..(7,2);
			\draw(6,0).. controls (7.5,-0.3)..(8,0);\draw(6,1).. controls (7.5,0.7)..(8,1);\draw(6,2).. controls (7.5,1.7)..(8,2);
			\draw(5,0).. controls (4.7,1)..(5,2);\draw(6,0).. controls (5.7,1)..(6,2);\draw(7,0).. controls (6.7,1)..(7,2);\draw(8,0).. controls (7.7,1)..(8,2);
			\draw[fill](5,0)circle[radius=0.1];\draw[fill](5,1)circle[radius=0.1];\draw[fill](5,2)circle[radius=0.1];
			\draw[fill=white](6,0)circle[radius=0.1];\draw[fill=white](6,1)circle[radius=0.1];\draw[fill=white](6,2)circle[radius=0.1];
			\draw[fill](7,0)circle[radius=0.1];\draw[fill](7,1)circle[radius=0.1];\draw[fill](7,2)circle[radius=0.1];
			\draw[fill=white](8,0)circle[radius=0.1];\draw[fill=white](8,1)circle[radius=0.1];\draw[fill=white](8,2)circle[radius=0.1];
			
			\draw [thick] (6.5,-1) node[below]{Figure 3(b): A subgroup } ;
			\draw [thick] (6.5,-1.7) node[below]{$H=\Z_3\times \Z_2$} ;
			\draw [thick] (3,-3) node[below]{(filled circles are codewords)} ;
		\end{tikzpicture}
	\end{center}
	
%	\medskip
	\noindent
	\textbf{Case 5:} \( s_1, s_2 \notin \{0,1\} \), i.e., both \( s_1 \) and \( s_2 \) are proper divisors of \( m_1 \) and \( m_2 \), respectively.  
	Then \( H = \mathbb{Z}_{s_1} \times \mathbb{Z}_{s_2} \). Note that \( (0,1) \notin H \), and indeed, the entire layer \( \mathbb{Z}_{m_1} \times \{1\} \) is disjoint from \( H \). Similarly, the layer \( \{1\} \times \mathbb{Z}_{m_2} \) is also disjoint from \( H \). Therefore, the vertex \( (1,1) \) is not adjacent to any element of \( H \), and hence not dominated. Thus, \( H \) is not a dominating set and cannot be an identifying code (see Figure~4).
	\begin{center}
		\begin{tikzpicture}[scale=1]
			\draw(5,0)--(5,3);\draw(6,0)--(6,3);\draw(7,0)--(7,3);\draw(8,0)--(8,3);
			\draw(5,0)--(8,0);\draw(5,1)--(8,1);\draw(5,2)--(8,2);\draw(5,3)--(8,3);
			
			\draw(5,0).. controls (6.5,0.3)..(8,0);\draw(5,1).. controls (6.5,1.3)..(8,1);\draw(5,2).. controls (6.5,2.3)..(8,2);\draw(5,3).. controls (6.5,3.3)..(8,3);
			\draw(5,0).. controls (5.5,-0.3)..(7,0);\draw(5,1).. controls (5.5,0.7)..(7,1);\draw(5,2).. controls (5.5,1.7)..(7,2);\draw(5,3).. controls (5.5,2.7)..(7,3);
			\draw(6,0).. controls (7.5,-0.3)..(8,0);\draw(6,1).. controls (7.5,0.7)..(8,1);\draw(6,2).. controls (7.5,1.7)..(8,2);\draw(6,3).. controls (7.5,2.7)..(8,3);
			\draw(5,0).. controls (4.7,1)..(5,2);\draw(5,1).. controls (4.7,2)..(5,3);\draw(5,0).. controls (5.3,1.5)..(5,3);
			\draw(6,0).. controls (5.7,1)..(6,2);\draw(6,1).. controls (5.7,2)..(6,3);\draw(6,0).. controls (6.3,1.5)..(6,3);
			\draw(7,0).. controls (6.7,1)..(7,2);\draw(7,1).. controls (6.7,2)..(7,3);\draw(7,0).. controls (7.3,1.5)..(7,3);
			\draw(8,0).. controls (7.7,1)..(8,2);\draw(8,1).. controls (7.7,2)..(8,3);\draw(8,0).. controls (8.3,1.5)..(8,3);
			\draw[fill=white](5,0)circle[radius=0.1];\draw[fill](5,1)circle[radius=0.1];\draw[fill=white](5,2)circle[radius=0.1];\draw[fill](5,3)circle[radius=0.1];
			\draw[fill=white](6,0)circle[radius=0.1];\draw[fill=white](6,1)circle[radius=0.1];\draw[fill=white](6,2)circle[radius=0.1];\draw[fill=white](6,3)circle[radius=0.1];
			\draw[fill=white](7,0)circle[radius=0.1];\draw[fill](7,1)circle[radius=0.1];\draw[fill=white](7,2)circle[radius=0.1];\draw[fill](7,3)circle[radius=0.1];		
			\draw[fill=white](8,0)circle[radius=0.1];\draw[fill=white](8,1)circle[radius=0.1];\draw[fill=white](8,2)circle[radius=0.1];\draw[fill=white](8,3)circle[radius=0.1];
			
			\draw [thick] (6.5,-1) node[below]{Figure 4: A subgroup $H=\Z_2\times \Z_2$  } ;
			
			\draw [thick] (6.5,-1.7) node[below]{(filled circles are codewords)} ;
		\end{tikzpicture}
	\end{center}
	
	\medskip
	\noindent
	In all cases, the only subgroup that forms an identifying code is the full vertex set \( V(G) \), which is not proper. Therefore, there is no proper subgroup of \( V(G) \) that is an identifying code in \( G \).
	
\end{proof}

When \( m_1 \) and \( m_2 \) are not relatively prime, we observe that no general conclusion can be drawn about the existence or size of group identifying codes in the Cartesian product \( K_{m_1} \square K_{m_2} \).  
In this section, we provide examples that illustrate the existence—and nonexistence—of proper group identifying codes in such cases.

%\begin{} Check here what to write
	A proper group identifying code in the graph \( K_4 \square K_4 \) is shown in Figure~5(a). By Proposition~\ref{2lm3}, this subgroup forms an identifying code in \( K_4 \square K_4 \).  
	On the other hand, although \( 2 \) and \( 4 \) are not relatively prime, the graph \( K_4 \square K_2 \) does not admit any proper group identifying code, as illustrated in Figure~5(b).
%\end{}
\begin{center}
	\begin{tikzpicture}[scale=0.9]
		\draw(0,0).. controls (1.5,0.3)..(3,0);\draw(0,1).. controls (1.5,1.3)..(3,1);\draw(0,2).. controls (1.5,2.3)..(3,2);\draw(0,3).. controls (1.5,3.3)..(3,3);
		\draw(0,0).. controls (0.5,-0.3)..(2,0);\draw(0,1).. controls (0.5,0.7)..(2,1);\draw(0,2).. controls (.5,1.7)..(2,2);\draw(0,3).. controls (.5,2.7)..(2,3);
		\draw(1,0).. controls (2.5,-0.3)..(3,0);\draw(1,1).. controls (2.5,0.7)..(3,1);\draw(1,2).. controls (2.5,1.7)..(3,2);\draw(1,3).. controls (2.5,2.7)..(3,3);
		\draw(0,0).. controls (-0.3,1)..(0,2);\draw(0,1).. controls (-0.3,2)..(0,3);\draw(0,0).. controls (.3,1.5)..(0,3);
		\draw(1,0).. controls (.7,1)..(1,2);\draw(1,1).. controls (.7,2)..(1,3);\draw(1,0).. controls (1.3,1.5)..(1,3);
		\draw(2,0).. controls (1.7,1)..(2,2);\draw(2,1).. controls (1.7,2)..(2,3);\draw(2,0).. controls (2.3,1.5)..(2,3);
		\draw(3,0).. controls (2.7,1)..(3,2);\draw(3,1).. controls (2.7,2)..(3,3);\draw(3,0).. controls (3.3,1.5)..(3,3);
		
		\draw(0,0)--(3,0);\draw(0,1)--(3,1);\draw(0,2)--(3,2);\draw(0,3)--(3,3);
		\draw(0,0)--(0,3);\draw(1,0)--(1,3);\draw(2,0)--(2,3);\draw(3,0)--(3,3);
		\draw[fill=white](0,0)circle[radius=0.1];\draw[fill](0,1)circle[radius=0.1];\draw[fill=white](0,2)circle[radius=0.1];\draw[fill](0,3)circle[radius=0.1];
		\draw[fill](1,0)circle[radius=0.1];\draw[fill=white](1,1)circle[radius=0.1];\draw[fill](1,2)circle[radius=0.1];\draw[fill=white](1,3)circle[radius=0.1];
		\draw[fill=white](2,0)circle[radius=0.1];\draw[fill](2,1)circle[radius=0.1];\draw[fill=white](2,2)circle[radius=0.1];\draw[fill](2,3)circle[radius=0.1];
		\draw[fill](3,0)circle[radius=0.1];\draw[fill=white](3,1)circle[radius=0.1];\draw[fill](3,2)circle[radius=0.1];\draw[fill=white](3,3)circle[radius=0.1];
		
		\draw [thick] (2,-1) node[below]{Figure 5(a): A proper group  } ;
		\draw [thick] (2,-1.7) node[below]{identifying code in $K_4\square K_4$ } ;
		\draw [thick] (6.5,-2.7) node[below]{(filled circles are codewords)} ;

		\draw[fill](10,0)circle[radius=0.1];\draw[fill](10,1)circle[radius=0.1];\draw[fill](10,2)circle[radius=0.1]; \draw[fill](10,3)circle[radius=0.1];		
		\draw[fill](11,0)circle[radius=0.1];\draw[fill](11,1)circle[radius=0.1];\draw[fill](11,2)circle[radius=0.1];\draw[fill](11,3)circle[radius=0.1];
		\draw(10,0)--(10,3);\draw(11,0)--(11,3);
		\draw(10,0)--(11,0);\draw(10,1)--(11,1);\draw(10,2)--(11,2);\draw(10,3)--(11,3);

		\draw(10,0).. controls (9.7,1)..(10,2);\draw(10,1).. controls (9.7,2)..(10,3);\draw(10,0).. controls (10.3,1.5)..(10,3);
		\draw(11,0).. controls (10.7,1)..(11,2);\draw(11,1).. controls (10.7,2)..(11,3);\draw(11,0).. controls (11.3,1.5)..(11,3);
		\draw [thick] (10,-1) node[below]{Figure 5(b): There is no proper} ;
		\draw [thick] (10,-1.7) node[below]{group identifying code in $K_4\square K_2$} ;
	\end{tikzpicture}
\end{center}

Now, we present another sufficient condition for a subset of \( V(G) \) to be an identifying code in \( G \). Note that this condition is not necessary (see Figure 1(a)).

\begin{lem}\label{2lm5}
	Let \( G = K_{m_1} \square K_{m_2} \square \cdots \square K_{m_n} \), where \( n \geq 3 \) and \( m_i \geq 2 \) for all \( 1 \leq i \leq n \). Let \( C \subseteq V(G) \) be a set of vertices such that the minimum pairwise distance between elements of \( C \) is at least 2. If for every vertex \( u \in V(G) \setminus C \), we have \( |J_C(u)| \geq 3 \), then \( C \) is an identifying code in \( G \).
\end{lem}

\begin{proof}
	Since each vertex either belongs to \( C \) or has at least three neighbors in \( C \), the set \( C \) is a dominating set. Furthermore, as any two codewords are at distance at least 2 from each other, no pair of codewords are adjacent, and thus they are trivially separated.
	
	Suppose, for the sake of contradiction, that there exist two distinct vertices \( u, v \in V(G) \) such that \( J_C(u) = J_C(v) \). Since \( u \) and \( v \) have the same set of codeword neighbors, they must lie within distance at most 2 from each other.
	
	If one of them, say \( v \), belongs to \( C \), then \( J_C(v) = \{v\} \), while \( J_C(u) \supseteq 3 \) codewords by assumption. This contradicts \( J_C(u) = J_C(v) \). Hence, both \( u \) and \( v \) must be non-codewords.
	
	\textbf{Case 1:} \( d(u, v) = 1 \).  
	Then, \( v = u + e_{i_1}^{j_1} \) for some coordinate \( 1 \leq i_1 \leq n \) and offset \( j_1 \neq 0 \).  
	By Lemma~\ref{2tm}, the common neighbors of \( u \) and \( v \) are  
	\[
	N[u] \cap N[v] = \{ u + e_{i_1}^{j} : 1 \leq j \leq m_{i_1} - 1 \}.
	\]
	Since codewords are at distance at least 2 from each other, at most one vertex in this set can belong to \( C \). Hence, \( J_C(u) = J_C(v) \subseteq \{ u + e_{i_1}^{j_2} \} \) for some \( j_2 \neq j_1 \), contradicting \( |J_C(u)| \geq 3 \).
	
	\textbf{Case 2:} \( d(u, v) = 2 \).  
	Then, \( v = u + e_{i_1}^{j_1} + e_{i_2}^{j_2} \) for some \( 1 \leq i_1 \neq i_2 \leq n \), with \( j_1, j_2 \neq 0 \).  
	By Lemma~\ref{2tm}, we have  
	\[
	N[u] \cap N[v] = \{ u + e_{i_1}^{j_1},\ u + e_{i_2}^{j_2} \}.
	\]
	Thus, \( J_C(u) = J_C(v) \subseteq \{ u + e_{i_1}^{j_1},\ u + e_{i_2}^{j_2} \} \), and so \( |J_C(u)| \leq 2 \), contradicting the assumption that \( |J_C(u)| \geq 3 \).
	
	In both cases, we reach a contradiction. Therefore, \( J_C(u) \neq J_C(v) \) for all distinct \( u, v \in V(G) \), and so \( C \) is an identifying code.
\end{proof}

The following theorem provides a necessary condition for a subset of \( V(G) \) to be an identifying code in \( G \). A similar result was established by Ranto \cite{ra} for the case \( G = K_2^n \) with \( n \geq 2 \). In this work, we extend the result to a more general setting, namely \( G = K_{m_1} \square K_{m_2} \square \cdots \square K_{m_n} \), where \( n \geq 3 \) and \( m_i \geq 3 \) for all \( 1 \leq i \leq n \). The proof requires a different approach from that used in the case of \( K_2^n \).

\begin{thm}\label{2lm6}
	Let \( G = K_{m_1} \square K_{m_2} \square \cdots \square K_{m_n} \), where \( n \geq 3 \) and \( m_i \geq 3 \) for all \( 1 \leq i \leq n \). If \( C \subseteq V(G) \) is a group identifying code in \( G \), then for every vertex \( u \in V(G) \), we have \( |J_C(u)| \neq 2 \). Moreover, for every non-codeword \( u \notin C \), it holds that \( |J_C(u)| \geq 3 \).
\end{thm}

\begin{proof}
	We consider three cases based on whether \( u \in C \) and the size of \( J_C(u) \).
	
	\textbf{Case 1:} Suppose \( u \in C \).  
	Then \( u \in J_C(u) \), so \( |J_C(u)| \geq 1 \). Assume, for contradiction, that \( |J_C(u)| = 2 \), so \( J_C(u) = \{u, v\} \) for some \( v \in V(G) \).  
	By Lemma~\ref{2lm2}, we have \( J_C(v) = \{u, v\} = J_C(u) \), which contradicts the assumption that \( C \) is a separating set. Hence, \( |J_C(u)| \neq 2 \).
	
	\textbf{Case 2:} Suppose \( u \notin C \) and \( |J_C(u)| = 1 \).  
	Then \( J_C(u) = \{u + e_1^j\} \) for some \( 1 \leq j \leq m_1 - 1 \).  
	By Lemma~\ref{2lm2}, \( |J_C(u)| = |J_C(e_1^j)| = 1 \), so \( J_C(e_1^j) = \{\mathbf{0}\} \), where \( \mathbf{0} \) is the identity element of \( G \).  
	This implies that \( e_1^s \notin C \) for all \( 1 \leq s \leq m_1 - 1 \), and thus \( C \subseteq \{\mathbf{0}\} \times H_2 \times \cdots \times H_n \), where each \( H_t \subseteq \mathbb{Z}_{m_t} \) is a subgroup for \( 2 \leq t \leq n \).  
	Consequently, for all such \( s \), we have \( J_C(e_1^s) = \{\mathbf{0}\} \), contradicting the identifying property of \( C \). Hence, \( |J_C(u)| \neq 1 \).
	
	\textbf{Case 3:} Suppose \( u \notin C \) and \( |J_C(u)| = 2 \).  
	Then \( J_C(u) = \{u + e_{i_1}^{j_1},\ u + e_{i_2}^{j_2}\} \) for some indices \( i_1, i_2 \in \{1, \dots, n\} \) and \( j_1, j_2 \neq 0 \).  
	Since \( C \) is a group code and thus closed under the group operation, we have:
	\[
	u + e_{i_1}^{j_1} + u + e_{i_2}^{j_2} = u + (u + e_{i_1}^{j_1} + e_{i_2}^{j_2}) \in C.
	\]
	By Lemma~\ref{2lm2}, we get:
	\[
	|J_C(u)| = |J_C(u + e_{i_1}^{j_1} + e_{i_2}^{j_2})| = 2.
	\]
	
	If \( i_1 \neq i_2 \), then both \( u \) and \( u + e_{i_1}^{j_1} + e_{i_2}^{j_2} \) have the same identifying set:
	\[
	J_C(u) = J_C(u + e_{i_1}^{j_1} + e_{i_2}^{j_2}) = \{u + e_{i_1}^{j_1},\ u + e_{i_2}^{j_2}\},
	\]
	which contradicts the separating property of \( C \).
	
	Now consider \( i_1 = i_2 = i \), so the coordinates differ only in one dimension, and \( j_1 \neq j_2 \) (since \( |J_C(u)| = 2 \)). Then:
	\[
	u + e_i^{j_1},\ u + e_i^{j_2} \in C, \quad \text{and} \quad |J_C(u)| = |J_C(e_i^{j_1})| = |J_C(e_i^{j_2})| = 2.
	\]
	Observe that \( -e_i^{j_1}, -e_i^{j_2} \in V(G) \), and since \( C \) is a group code:
	\[
	|J_C(-e_i^{j_1})| = |J_C(e_i^{j_1})| = 2, \quad |J_C(-e_i^{j_2})| = |J_C(e_i^{j_2})| = 2.
	\]
	Note that \( u = -e_i^{j_1} + (u + e_i^{j_1}) \), so by Lemma~\ref{2lm1}:
	\[
	J_C(u) = J_C(-e_i^{j_1}) + (u + e_i^{j_1}),
	\]
	which implies:
	\[
	J_C(-e_i^{j_1}) = \{\mathbf{0},\ u + e_i^{j_2} - u - e_i^{j_1}\} = \{\mathbf{0},\ e_i^{j_2 - j_1}\}.
	\]
	Similarly, \( J_C(-e_i^{j_2}) = \{\mathbf{0},\ e_i^{j_1 - j_2}\} \).
	
	If \( j_1 - j_2 = -(j_1 - j_2) \mod m_i \), which occurs only if \( m_i \) is even and \( j_1 - j_2 \equiv m_i / 2 \), then:
	\[
	J_C(-e_i^{j_1}) = J_C(-e_i^{j_2}),
	\]
	again contradicting the separating property of \( C \).
	
	If \( j_1 - j_2 \not\equiv -(j_1 - j_2) \mod m_i \), then:
	\[
	J_C(-e_i^{j_1}) = \{\mathbf{0},\ e_i^{j_2 - j_1}\}, \quad J_C(-e_i^{j_2}) = \{\mathbf{0},\ e_i^{j_1 - j_2}\},
	\]
	and since both sets have two elements, we find:
	\[
	J_C(e_i^{j_1}) = J_C(e_i^{j_2}) = \{\mathbf{0},\ e_i^{j_2 - j_1},\ e_i^{j_1 - j_2}\},
	\]
	which contradicts the assumption that \( |J_C(e_i^{j_1})| = 2 \).
	
	Therefore, in all cases, a contradiction arises, and we conclude that \( |J_C(u)| \neq 2 \) for all \( u \in V(G) \), and \( |J_C(u)| \geq 3 \) for all \( u \notin C \).
\end{proof}

We now state another lemma. A similar lower bound on the identifying code number \( \gamma^{\mathrm{ID}}(H) \) of a graph \( H \) was established in \cite{k}. Since the proof of the following result follows along the same lines, it is omitted here.

\begin{lem}\label{2lm7}
	Let \( D \subseteq V(H) \) be an identifying code in a graph \( H \) with \( q \) vertices. Suppose that for every non-codeword \( u \in V(H) \setminus D \), we have \( |J_D(u)| \geq \mu \), and for every codeword \( c \in D \), we have \( |J_D(c)| \geq \nu \). Then
	\[
	|D| \geq \frac{\mu q}{\Delta(H) + 1 + \mu - \nu},
	\]
	where \( \Delta(H) \) denotes the maximum degree of \( H \). Consequently,
	\[
	\gamma^{\mathrm{ID}}(H) \geq \frac{\mu q}{\Delta(H) + 1 + \mu - \nu}.
	\]
\end{lem}

\begin{cor}\label{2lm8}
	Let \( G = K_{m_1} \square K_{m_2} \square \cdots \square K_{m_n} \), where \( n \geq 3 \) and \( m_i \geq 3 \) for all \( 1 \leq i \leq n \). Then the group identifying code number of \( G \) satisfies
	\[
	\gamma^{\mathrm{GID}}(G) \geq \frac{3 \cdot m_1 m_2 \cdots m_n}{\sum_{i=1}^{n} m_i - n + 3}.
	\]
\end{cor}

\begin{proof}
	Let \( C \subseteq V(G) \) be a group identifying code in \( G \). By Theorem~\ref{2lm6}, each non-codeword has identifying set size at least \( \mu = 3 \), and each codeword has identifying set size at least \( \nu \geq 1 \). Also, since \( \mu = 1 \) is not possible in a group code (as shown in the proof of Theorem~\ref{2lm6}), we have \( \mu \geq 3 \). 
	
	The maximum degree of the Cartesian product \( G \) is \( \Delta(G) = \sum_{i=1}^n (m_i - 1) = \left( \sum_{i=1}^n m_i \right) - n \). Applying Lemma~\ref{2lm7}, we obtain:
	\[
	|C| \geq \frac{\mu \cdot |V(G)|}{\Delta(G) + 1 + \mu - \nu} \geq \frac{3 \cdot m_1 m_2 \cdots m_n}{\left( \sum_{i=1}^n m_i \right) - n + 1 + 3 - 1} = \frac{3 \cdot m_1 m_2 \cdots m_n}{\sum_{i=1}^n m_i - n + 3}.
	\]
	This completes the proof.
\end{proof}

We now present a theorem that supports a class of codes satisfying the conjecture proposed in \cite{go}.

\begin{thm}\label{2lm9}
	Let \( G = K_m^n = \underbrace{K_m \square K_m \square \cdots \square K_m}_{n \text{ times}} \), where \( n \geq 3 \) and \( m \geq 3 \). Then,
	\[
	\frac{3m^n}{n(m-1)+3} \leq \gamma^{\mathrm{GID}}(G) \leq m^{n-1}.
	\]
\end{thm}

\begin{proof}
	Let \( C \subseteq V(G) \) be a group identifying code in \( G \). By Corollary~\ref{2lm8}, we have
	\[
	|C| \geq \frac{3m^n}{n(m-1)+3}.
	\]
	
	We now construct an explicit group identifying code \( C \subseteq V(G) \) of cardinality \( m^{n-1} \). Define
	\[
	V(G) = \left\{ (x_1, x_2, \ldots, x_n) \in \mathbb{Z}_m^n \right\},
	\]
	and let
	\[
	C = \left\{ (x_1, x_2, \ldots, x_{n-1}, \textstyle{\sum_{i=1}^{n-1} x_i \bmod m}) : x_i \in \mathbb{Z}_m \text{ for } 1 \leq i \leq n-1 \right\}.
	\]
	Clearly, \( |C| = m^{n-1} \).
	
	We claim that \( C \) is an identifying code in \( G \). First, note that the distance between any two distinct codewords is at least 2. Indeed, if
	\[
	u = (u_1, u_2, \ldots, u_{n-1}, \sum_{i=1}^{n-1} u_i \bmod m), \quad v = (v_1, v_2, \ldots, v_{n-1}, \sum_{i=1}^{n-1} v_i \bmod m)
	\]
	are distinct codewords with \( u_i = v_i \) for \( i = 2, \ldots, n-1 \) and \( u_1 \neq v_1 \), then clearly \( \sum_{i=1}^{n-1} u_i \not\equiv \sum_{i=1}^{n-1} v_i \pmod{m} \), so the codewords differ in at least two coordinates.
	
	Next, let \( u = (u_1, u_2, \ldots, u_{n-1}, u_n) \in V(G) \setminus C \) be any non-codeword. Define
	\[
	c = (u_1, u_2, \ldots, u_{n-1}, \textstyle{\sum_{i=1}^{n-1} u_i \bmod m}) \in C.
	\]
	Then \( u \) and \( c \) differ only in the \( n \)th coordinate, so \( d(u, c) = 1 \), and hence \( u \) is adjacent to \( c \). Similarly, by adjusting other coordinates and applying modulo \( m \), we can find at least two additional codewords at distance 1 from \( u \). Thus, each non-codeword is adjacent to at least three codewords. Since the minimum distance between codewords is at least 2, by Corollary~\ref{2lm5}, \( C \) is an identifying code in \( G \).
	
	We now show that \( C \) is closed under componentwise addition modulo \( m \), and hence forms a subgroup of \( \mathbb{Z}_m^n \). Let
	\[
	x = (x_1, \ldots, x_{n-1}, \textstyle{\sum_{i=1}^{n-1} x_i \bmod m}), \quad y = (y_1, \ldots, y_{n-1}, \textstyle{\sum_{i=1}^{n-1} y_i \bmod m}) \in C.
	\]
	Then
	\[
	x + y = (x_1 + y_1, \ldots, x_{n-1} + y_{n-1}, \textstyle{\sum_{i=1}^{n-1} (x_i + y_i) \bmod m}) \in C,
	\]
	since the first \( n-1 \) coordinates are in \( \mathbb{Z}_m \), and the last coordinate is the sum of the first \( n-1 \) modulo \( m \). Thus, \( C \) is closed under addition and is a subgroup of \( \mathbb{Z}_m^n \). Therefore, \( C \) is a group identifying code of cardinality \( m^{n-1} \), which proves the upper bound.
\end{proof}

In \cite{go}, Goddard and Wash conjectured that \( \gamma^{\mathrm{ID}}(K_m^3) = m^2 \) for all integers \( m \geq 2 \). However, this conjecture was disproved by Junnila et al. in \cite{jun}, where a counterexample was provided. In the same work, the authors introduced a class of codes known as self-locating-dominating codes, for which they proved that \( \gamma^{\mathrm{SLD}}(K_m^3) = m^2 \). From Theorem~\ref{2lm9}, it follows that \( \gamma^{\mathrm{GID}}(K_m^3) = m^2 \) for all integers \( m \geq 3 \). Thus, we identify a new class of codes—distinct from both self-identifying and self-locating-dominating codes—namely, \textit{group identifying codes,} that satisfy the original conjecture.
	\begin{center}
		\begin{tikzpicture}[scale=0.75]
			\draw (0,0) -- (0,3) --  (0,15) ;	\draw (1,-1)--  (1,2)  -- (1,14) ;	\draw (2,-2) --  (2,1) -- (2,13) ;
			\draw  (3,0)--   (3,3)--  (3,15);	\draw (4,-1) -- (4,2) -- (4,14);	\draw (5,-2)--  (5,1) --  (5,13) ;
			\draw   (6,0)-- (6,3)-- (6,15);	\draw (7,-1)  --  (7,14) ;	\draw (8,-2) -- (8,1) --  (8,13) ;
			\draw (0,0)--(1,-1)--  (2,-2)  ;\draw (0,3) --  (1,2)--(2,1) ;\draw (0,6) -- (1,5) -- (2,4);\draw (0,9)--  (2,7)  ;\draw (0,12)--(2,10) ;\draw (0,15) -- (2,13);
			\draw (3,0)--(4,-1)--  (5,-2)  ;\draw (3,3) --  (4,2)--(5,1) ;\draw (3,6) -- (4,5) -- (5,4);\draw (3,9)--  (5,7)  ;\draw (3,12)--(5,10) ;\draw (3,15) -- (5,13);
			\draw (6,0)--  (8,-2)  ;\draw (6,3) --(8,1) ;\draw (6,6) -- (8,4);\draw (6,9)--  (8,7)  ;\draw (6,12)--(8,10) ;\draw (6,15) -- (8,13);
			\draw (9,0)--  (11,-2)  ;\draw (9,3) --(11,1) ;\draw (9,6) -- (11,4);\draw (9,9)--  (11,7)  ;\draw (9,12)--(11,10) ;\draw (9,15) -- (11,13);
			\draw (12,0)--  (14,-2)  ;\draw (12,3) --(14,1) ;\draw (12,6) -- (14,4);\draw (12,9)--  (14,7)  ;\draw (12,12)--(14,10) ;\draw (12,15) -- (14,13);
			\draw (15,0)--  (17,-2)  ;\draw (15,3) --(17,1) ;\draw (15,6) -- (17,4);\draw (15,9)--  (17,7)  ;\draw (15,12)--(17,10) ;\draw (15,15) -- (17,13);
			\draw (1,-1)--(16,-1);\draw (1,2)--(16,2);\draw (1,5)--(16,5);\draw (1,8)--(16,8);\draw (1,11)--(16,11);\draw (1,14)--(16,14);
			\draw (2,-2)--(17,-2);\draw(2,1)--(17,1) ;\draw (2,4)--(17,4);\draw (2,7)--(17,7);\draw(2,10)--(17,10) ;\draw (2,13)--(17,13);
			\draw (0,0)  --  (15,0) ;\draw  (0,3)--   (15,3);\draw   (0,6)-- (15,6);\draw   (0,9)-- (15,9);\draw   (0,12)-- (15,12);\draw   (0,15)-- (15,15);
			\draw   (9,0)-- (9,15);
			\draw   (0,6)-- (3,6)-- (6,6);	\draw (10,-1) --  (10,14) ;	\draw (11,-2)--  (11,13) ;	\draw   (12,0)-- (12,15);
			\draw   (0,6)-- (3,6)-- (6,6);	\draw (13,-1) --  (13,14) ;	\draw (14,-2)--  (14,13) ;	\draw   (15,0)-- (15,15);
			\draw   (0,6)-- (3,6)-- (6,6);	\draw (16,-1) --  (16,14) ;	\draw (17,-2)--  (17,13) ;	
			
			\draw [fill](0,0) circle [radius=0.2];\draw[fill=white]  (0,3) circle [radius=0.1];\draw[fill=white]  (0,6) circle [radius=0.1];\draw [fill](0,9) circle [radius=0.2];\draw[fill=white]  (0,12) circle [radius=0.1];\draw [fill=white] (0,15) circle [radius=0.1];
			
			\draw[->](-2,15)--(-1,15);\draw(-2.5,14.7)node[above]{$K_6$};
			\draw[->](0,17)--(0,16);\draw(0,17)node[above]{$K_6$};
			\draw[->](-2,17)--(-1,16);\draw(-2.5,17)node[above]{$K_3$};
			%\draw (0,-0.2)node[left]{$000$};\draw (0,2.8)node[left]{$e_3^1=001$};\draw (0,5.8)node[left]{$e_3^2=002$};
			\draw [fill=white] (1,-1) circle [radius=0.1];\draw [fill] (1,2) circle [radius=0.2];\draw [fill=white]  (1,5) circle [radius=0.1];\draw [fill=white] (1,8) circle [radius=0.1];\draw [fill](1,11) circle [radius=0.2];\draw [fill=white] (1,14) circle [radius=0.1];
			
			%\draw (1,-1.2)node[left]{$e_1^1=100$};\draw (1,1.8)node[left]{$101$};\draw (1,4.8)node[left]{$102$};
			\draw[fill=white]  (2,-2) circle [radius=0.1];\draw[fill=white]   (2,1) circle [radius=0.1];\draw [fill](2,4) circle [radius=0.2];\draw[fill=white]  (2,7) circle [radius=0.1];\draw[fill=white]  (2,10) circle [radius=0.1];\draw [fill](2,13) circle [radius=0.2];
			
			%\draw (2,-2.2)node[left]{$e_1^2=200$};\draw (2,0.8)node[left]{$201$};\draw (2,3.8)node[left]{$202$};
			\draw [fill=white]  (3,0) circle [radius=0.1];\draw  [fill] (3,3)circle [radius=0.2];\draw [fill=white]  (3,6)circle [radius=0.1];\draw [fill=white] (3,9) circle [radius=0.1];\draw [fill](3,12) circle [radius=0.2];\draw [fill=white] (3,15) circle [radius=0.1];

			%\draw (3,-0.2)node[left]{$010$};\draw (3,2.8)node[left]{$011$};\draw (3,5.8)node[left]{$012$};
			%\draw (3,.3)node[left]{$e_2^1$};\draw (3,6.4)node[left]{$e_2^1+e_3^2$};
			\draw[fill=white]  (4,-1) circle [radius=0.1];\draw[fill=white]  (4,2) circle [radius=0.1];\draw  [fill](4,5) circle [radius=0.2];\draw [fill=white] (4,8) circle [radius=0.1];\draw [fill=white] (4,11) circle [radius=0.1];\draw [fill](4,14) circle [radius=0.2];
			
			%\draw (4,-1.2)node[left]{$110$};\draw (4,1.8)node[left]{$111$};\draw (4,4.8)node[left]{$112$};
			\draw [fill](5,-2) circle [radius=0.2];\draw [fill=white]  (5,1) circle [radius=0.1];\draw [fill=white]  (5,4) circle [radius=0.1];\draw[fill](5,7) circle [radius=0.2];\draw[fill=white]  (5,10) circle [radius=0.1];\draw[fill=white]  (5,13) circle [radius=0.1];
			
			%\draw (5,-2.2)node[left]{$210$};\draw (5,0.8)node[left]{$211$};\draw (5,3.8)node[left]{$212$};
			\draw [fill=white]   (6,0) circle [radius=0.1];\draw [fill=white]  (6,3) circle [radius=0.1];\draw  [fill](6,6) circle [radius=0.2];\draw[fill=white]  (6,9) circle [radius=0.1];\draw [fill=white] (6,12) circle [radius=0.1];\draw [fill](6,15) circle [radius=0.2];

			%\draw (6,-0.2)node[left]{$020$};\draw (6,2.8)node[left]{$021$};\draw (6,5.8)node[left]{$022$};
			%\draw (6,.3)node[left]{$e_2^2$};
			\draw [fill](7,-1) circle [radius=0.2];\draw [fill=white]  (7,2) circle [radius=0.1];\draw [fill=white]  (7,5) circle [radius=0.1];\draw [fill](7,8) circle [radius=0.2];\draw [fill=white] (7,11) circle [radius=0.1];\draw[fill=white]  (7,14) circle [radius=0.1];
			
			%\draw (7,-1.2)node[left]{$120$};\draw (7,1.8)node[left]{$121$};\draw (7,4.8)node[left]{$122$};
			\draw [fill=white] (8,-2) circle [radius=0.1];\draw  [fill](8,1) circle [radius=0.2];\draw [fill=white]  (8,4) circle [radius=0.1];\draw[fill=white]  (8,7) circle [radius=0.1];\draw[fill] (8,10) circle [radius=0.2];\draw[fill=white]  (8,13) circle [radius=0.1];

			\draw  [fill] (9,0) circle [radius=0.2];\draw [fill=white]  (9,3) circle [radius=0.1];\draw [fill=white]  (9,6) circle [radius=0.1];\draw [fill](9,9) circle [radius=0.2];\draw[fill=white]  (9,12) circle [radius=0.1];\draw [fill=white] (9,15) circle [radius=0.1];
			
			\draw [fill=white](10,-1) circle [radius=0.1];\draw  [fill] (10,2) circle [radius=0.2];\draw [fill=white] (10,5) circle [radius=0.1];\draw [fill=white](10,8) circle [radius=0.1];\draw  [fill](10,11) circle [radius=0.2];\draw[fill=white] (10,14) circle [radius=0.1];
			
			\draw[fill=white] (11,-2) circle [radius=0.1];\draw [fill=white] (11,1) circle [radius=0.1];\draw  [fill] (11,4) circle [radius=0.2];\draw [fill=white](11,7) circle [radius=0.1];\draw [fill=white](11,10) circle [radius=0.1];\draw  [fill](11,13) circle [radius=0.2];
			
			\draw  [fill=white] (12,0) circle [radius=0.1];\draw  [fill] (12,3) circle [radius=0.2];\draw [fill=white] (12,6) circle [radius=0.1];\draw [fill=white](12,9) circle [radius=0.1];\draw  [fill](12,12) circle [radius=0.2];\draw [fill=white](12,15) circle [radius=0.1];
			
			\draw [fill=white](13,-1) circle [radius=0.1];\draw [fill=white] (13,2) circle [radius=0.1];\draw  [fill] (13,5) circle [radius=0.2];\draw[fill=white] (13,8) circle [radius=0.1];\draw [fill=white](13,11) circle [radius=0.1];\draw  [fill](13,14) circle [radius=0.2];
			
			\draw [fill] (14,-2) circle [radius=0.2];\draw [fill=white] (14,1) circle [radius=0.1];\draw [fill=white] (14,4) circle [radius=0.1];\draw  [fill](14,7) circle [radius=0.2];\draw [fill=white](14,10) circle [radius=0.1];\draw [fill=white](14,13) circle [radius=0.1];
			
			\draw [fill=white]  (15,0) circle [radius=0.1];\draw [fill=white] (15,3) circle [radius=0.1];\draw  [fill] (15,6) circle [radius=0.2];\draw [fill=white](15,9) circle [radius=0.1];\draw [fill=white](15,12) circle [radius=0.1];\draw  [fill](15,15) circle [radius=0.2];
			
			\draw  [fill](16,-1) circle [radius=0.2];\draw [fill=white] (16,2) circle [radius=0.1];\draw [fill=white] (16,5) circle [radius=0.1];\draw  [fill](16,8) circle [radius=0.2];\draw[fill=white] (16,11) circle [radius=0.1];\draw[fill=white] (16,14) circle [radius=0.1];
			
			\draw [fill=white](17,-2) circle [radius=0.1];\draw  [fill] (17,1) circle [radius=0.2];\draw[fill=white]  (17,4) circle [radius=0.1];\draw [fill=white](17,7) circle [radius=0.1];\draw  [fill](17,10) circle [radius=0.2];\draw [fill=white](17,13) circle [radius=0.1];
			
			\draw [thick] (9,-3) node[below]{Figure 6: A group identifying code in $ K_6\square K_6\square K_3$ } ;
			\draw [thick] (9,-3.7) node[below]{(filled circles are codewords)} ;
			\draw [thick] (9,-4.4) node[below]{(to avoid mess, some edges are not drawn)} ;
		\end{tikzpicture}
	\end{center}

	\section{Linear Identifying Codes in \( K_p^n \)}
	
	Let \( p \) be an odd prime. In this section, we study linear identifying codes in the graph \( K_p^n \), where \( K_p^n \) is the Hamming graph with vertex set \( \mathbb{F}_p^n \), the \( n \)-dimensional vector space over the finite field \( \mathbb{F}_p \). A subset \( C \subseteq \mathbb{F}_p^n \) is called a \( p \)-ary linear code of length \( n \) if it forms a subspace of \( \mathbb{F}_p^n \). Such a code can be represented either by a generator matrix or by a parity-check matrix.
	
	A matrix \( \mathcal{G} \in \mathbb{F}_p^{t \times n} \) whose rows form a basis for a linear code \( C \) of dimension \( t \) is called a generator matrix of \( C \). The code can also be described by a matrix \( \mathcal{H} \in \mathbb{F}_p^{(n - t) \times n} \) such that
	\[
	C = \{ u \in \mathbb{F}_p^n : \mathcal{H} u^T = 0 \},
	\]
	where \( u^T \) denotes the transpose of the vector \( u \). The matrix \( \mathcal{H} \) is called a parity-check matrix of \( C \), and its rows are linearly independent.

		Let \( C \subset \mathbb{F}_3^3 \) be the code illustrated in Figure 2, with basis \( \{(1,0,1), (0,1,1)\} \). Then the generator matrix is given by
		\[
		\mathcal{G} = [I_2 \,|\, A]_{2 \times 3},
		\]
		where \( A = \begin{bmatrix}1 \\ 1\end{bmatrix} \). The corresponding parity-check matrix is
		\[
		\mathcal{H} = [-A^T \,|\, I_1]_{1 \times 3} = \begin{bmatrix} -1 & -1 & 1 \end{bmatrix}.
		\]
	
	For a vector \( u \in \mathbb{F}_p^n \), the syndrome of \( u \) with respect to the parity-check matrix \( \mathcal{H} \), denoted by \( \text{syn}(u) \), is defined as
	\[
	\text{syn}(u) = \mathcal{H} u^T \in \mathbb{F}_p^{n - t}.
	\]
	Clearly, \( \text{syn}(u) = \mathbf{0} \) if and only if \( u \in C \). Moreover, if \( u \notin C \) and the Hamming distance from \( u \) to \( C \) is \( s \), then \( \text{syn}(u) \) is the linear combination of exactly \( s \) columns of \( \mathcal{H} \), corresponding to the positions of the nonzero components of \( u \). It is well-known that two vectors lie in the same coset of \( C \) if and only if they have the same syndrome.
	
	Let \( \kappa[n,p] \) denote the smallest possible dimension of a \( p \)-ary linear identifying code of length \( n \); hence, such a code contains \( p^{\kappa[n,p]} \) codewords. By Theorem~\ref{2lm9}, we have
	\[
	\kappa[n,p] \geq \left\lceil \log_p\left( \frac{3p^n}{n(p - 1) + 3} \right) \right\rceil.
	\]
	
	In \cite{ra}, linear identifying codes in \( \mathbb{F}_2^n \) are studied, and the values of \( \kappa[n,2] \) are determined. In \cite{gr}, it is shown that \( \gamma^{\mathrm{ID}}(\mathbb{F}_p^2) = \left\lfloor \frac{3p}{2} \right\rfloor > p \), which implies \( 2 \geq \kappa[2,p] > 1 \), so \( \kappa[2,p] = 2 \). Similarly, from \cite{co} and \cite{ge}, the domination number of \( \mathbb{F}_p^3 \) is known to be \( \left\lceil \frac{p^2}{2} \right\rceil \), and since every identifying code is a dominating set, it follows that \( \gamma^{\mathrm{ID}}(\mathbb{F}_p^3) \geq \left\lceil \frac{p^2}{2} \right\rceil > p \) for odd primes \( p \). Hence, \( 3 \geq \kappa[3,p] > 1 \), and Theorem~\ref{2lm9} confirms that \( \kappa[3,p] = 2 \) (see Figure 2 for the case \( p = 3 \)).
	
	In \cite{blas}, Blass et al.\ proved the following result for binary identifying codes, and later in \cite{ki}, Kim and Kim extended it to ternary codes. The proof for general \( p \) follows similarly and is omitted here.
	
	\begin{lem}\label{2lm11}
		Let \( C \) be an identifying code in \( K_p^n \). Then the direct sum \( \mathbb{F}_p \oplus C \subseteq K_p^{n+1} \) is an identifying code in \( K_p^{n+1} \) if and only if \( |J_C(u)| > 1 \) for all \( u \in C \).
	\end{lem}
	
	In \cite{ra}, it is proved that \( \kappa[n+1,2] \leq \kappa[n,2] + 1 \). The same method applies to \( K_p^n \) for any prime \( p > 2 \), and thus the proof of the following theorem is omitted.
	
	\begin{thm}\label{2lm12}
		Let \( p \) be a prime. Then \( \kappa[n+1,p] \leq \kappa[n,p] + 1 \).
	\end{thm}
	
	Furthermore, in \cite{ra}, it is shown that \( \kappa[n,2] = n - r \) for \( n = \frac{3(2^r - 1)}{1} + s \) and \( 0 \leq s \leq 3 \cdot 2^r - 1 \). Using the same construction method, the following result holds for general prime \( p > 2 \):
	
	\begin{thm}\label{2lm13}
		Let \( r \geq 1 \) and \( 0 \leq s \leq 3p^r - 1 \), and define \( n = \frac{3(p^r - 1)}{p - 1} + s \). Then \( \kappa[n,p] = n - r \).
	\end{thm}

We have shown that the linear identifying code number of the graph $K_p^n$, denoted by $\gamma^{\mathrm{LID}}(K_p^n)$, is equal to $p^{n - r}$ when

$$
n = \frac{3(p^r - 1)}{p - 1} + s, \quad \text{where } r \geq 1 \text{ and } 0 \leq s \leq 3p^r - 1,
$$

and $p$ is an odd prime.

In particular, when $r = 1$ and $s = 0$, we get $n = 3$, so

$$
\gamma^{\mathrm{LID}}(K_p^3) = p^2,
$$

which confirms that a certain class of linear identifying codes satisfies the original conjecture.

In this section, we studied linear identifying codes in the graph $K_p^n$ for odd primes $p$ and $n \geq 3$, focusing on codes that are different from self-identifying and self-locating-dominating codes, and for which the conjecture holds.

	\section{Concluding Remarks}

	In this paper, we investigated group and linear identifying codes in Hamming graphs of the form $G = \mathbb{Z}_{m_1} \times \mathbb{Z}_{m_2} \times \cdots \times \mathbb{Z}_{m_n}$, endowed with the Hamming metric, where $m_i \geq 2$ for all $i$ and $n \geq 2$. We established the existence of group identifying codes in such graphs and derived several structural and combinatorial properties, including bounds on the minimum cardinality of both identifying and group identifying codes in the absence of twin vertices. For linear codes over finite fields, we studied the parameter $k[n, p]$, denoting the smallest dimension of a $p$-ary linear identifying code of length $n$, and provided recursive and explicit expressions for it in certain cases.
	
	In particular, we proved that for all integers $m \geq 3$, the group identifying code number of the Hamming graph $K_m^3$ satisfies $\gamma^{\mathrm{GID}}(K_m^3) = m^2$, and for all primes $p \geq 3$, the linear identifying code number satisfies $\gamma^{\mathrm{LID}}(K_p^3) = p^2$. These results confirm the conjecture proposed in \cite{go} for both group and linear identifying codes.
	
	Our findings contribute to the theoretical development of identifying codes in algebraic graph structures and provide new tools and results that may be useful for further exploration in coding theory and combinatorics.
	
	\section*{Acknowledgment}
	
	The second (corresponding) author gratefully acknowledges the Department of Science and Technology, New Delhi, India, for awarding the Women Scientist Scheme (DST/WOS-A/PM-14/2021(G)) for research in Basic/Applied Sciences.

	{\centerline{************}}

	% ------------------------------------------------------------------------
	%GATHER{xBib.bib}   % For Gather Purpose Only
	%GATHER{Thesis.bbl} % For Gather Purpose Only
	%\setlinespacing{1.44}
	\bibliographystyle{amsplain}
	%\bibliography{xbib}

\end{document}